\documentclass[a4paper]{article}

\usepackage{amsmath,amsfonts,amssymb,amsthm,mathtools,enumitem,hyperref,enumitem}

\setenumerate[1]{label={\arabic*)}}

\theoremstyle{plain}
\newtheorem{theo}{Theorem}[section]
\newtheorem{que}{Question}

\theoremstyle{definition}

\newtheorem*{exa*}{Example}
\newtheorem{lem}[theo]{Lemma}
\newtheorem{cor}[theo]{Corollary}
\newtheorem{pro}[theo]{Proposition}
\newtheorem*{pro*}{Proposition}

\theoremstyle{remark}
\newtheorem*{rem}{Remark}

\newcommand\tx\text
\newcommand\f\frac
\newcommand\D{\mathcal D}
\newcommand\ord{\mathcal O}
\newcommand\la\langle
\newcommand\ra\rangle
\newcommand\highl[1]{\textbf{#1}}
\newcommand\df[1]{\highl{#1}}
\newcommand\ind{1}
\newcommand\lh[1]{{#1}^{\mathrm l}}
\newcommand\rh[1]{{#1}^{\mathrm r}}
\newcommand\seq=
\newcommand\qes\eqqcolon

\begin{document}

\title{Almost-Orthogonality of Restricted Haar-Functions}
\author{Julian Weigt}
\date{Aalto University\\[2ex]\today}
\maketitle

\begin{abstract}
	We consider the Haar functions \(h_I\) on dyadic intervals. We show that if \(p>\f23\) and \(E\subset[0,1]\) then the set of all functions \(\|h_I\ind_E\|_2^{-1}h_I\ind_E\) with \(|I\cap E|\geq p|I|\) is a Riesz sequence. For \(p\leq\f23\) we provide a counterexample. 

\end{abstract}

\section{Introduction}

In this paper we prove a stability result for perturbed Haar functions. It grew out of the author's Master's thesis \cite{thesis}, written in Bonn under the supervision of Professor Christoph Thiele. It was motivated by an idea on how to to extend the result in \cite{kovac_thiele_zorin-kranich_2015} to three general functions.

The \df{Haar function} of the interval \(I\seq[a,b)\) is given by
\[h_I\seq-\ind_{\lh I}+\ind_{\rh I},\]
where \(\lh I\seq[a,\f{a+b}2)\) and \(\rh I\seq[\f{a+b}2,b)\) are the left and right halves of \(I\). We consider the Haar functions of the \df{dyadic intervals}
\[\D=\bigl\{\bigl[k2^n,(k+1)2^n\bigr)\bigm|n,k\in\mathbb{Z}\bigr\}.\]
The main result of this paper is the following theorem.
\begin{theo}\label{theo_arbeRiesz}
	For each \(p>\f23\) there is a constant \(c>0\) such that for all measurable sets \(E\subset[0,1)\) and all sequences \((a_I)_{I\in\D}\) with \(a_I=0\) if \(|I\cap E|<p|I|\), we have
	\begin{equation}\label{eq_goal}
		\Bigl\|\sum_{I\in\D}a_Ih_I\ind_E\Bigr\|_2^2\geq c\sum_{I\in\D}\|a_Ih_I\ind_E\|_2^2,
	\end{equation}
	whenever the right-hand side is finite. For \(p\leq\f23\) there is no such constant \(c>0\).
\end{theo}

\begin{rem} The proof strategy of Theorem \ref{theo_arbeRiesz} resembles the well known Bellman function technique as for example in \cite{nazarov1999bellman}. A rephrasing of the proof which resembles the Bellman function technique more closely can be found in Section 2.1.5 in \cite{thesis}.

	The proof also yields an explicit value for \(c\). We discuss its optimality in Section \ref{sec_remarks}. Furthermore, if the right-hand side of \eqref{eq_goal} converges, then the sum on the left-hand-side converges in \(L^2\) because for any finite subset \(D_0\subset\D\) we have
	\begin{equation}\label{eq_calcbessel}
		\Bigl\|\sum_{I\in D_0}a_Ih_I\ind_E\Bigr\|_2^2\leq\Bigl\|\sum_{I\in D_0}a_Ih_I\Bigr\|_2^2=\sum_{I\in D_0}\|a_Ih_I\|_2^2\leq\f1p\sum_{I\in D_0}\|a_Ih_I\ind_E\|_2^2.
	\end{equation}
	This also implies \eqref{eq_calcbessel} with \(D_0=\D\), which means that a reverse inequality of \eqref{eq_goal} holds as for all \(p>0\).
\end{rem}

In a more general setting, a sequence \(V\) in a Hilbert space is called a \df{Bessel sequence} if
\begin{align*}
	\Bigl\|\sum_{v\in V}a_vv\Bigr\|^2&\leq C\sum_{v\in V}|a_v|^2\\
	\intertext{holds, and a \df{Riesz sequence} if in addition also}
	\Bigl\|\sum_{v\in V}a_vv\Bigr\|^2&\geq c\sum_{v\in V}|a_v|^2
\end{align*}
holds, where the constants \(c>0\) and \(C<\infty\) respectively are independent of \((a_v)_{v\in V}\subset\ell^2(V)\). Inserting \(\f{a_I}{\|h_I\ind_E\|_2}\) for \(a_I\) in \eqref{eq_calcbessel} shows that for all \(p>0\), the sequence
\begin{equation}\label{eq_epbig}
	V=\biggl\{\f{h_I\ind_E}{\|h_I\ind_E\|_2}\biggm| I\in\D,\ |I\cap E|\geq p|I|\biggr\}
\end{equation}
is a Bessel sequence with constant \(\f1p\). Theorem \ref{theo_arbeRiesz} states that if \(p>\f23\) then \eqref{eq_epbig} is also a Riesz sequence. A weaker result already follows from the well-known \highl{Kadison-Singer Problem}, which was resolved recently by Marcus, Spielman and Srivastava in \cite{marcus2015interlacing}. In doing so, they also solved the numerous equivalent problems, one of which is the \highl{Feichtinger Conjecture}, which states that every Bessel sequence can be partitioned into finitely many Riesz sequences. This means it can already be concluded from \eqref{eq_calcbessel} that there is a finite partition of \eqref{eq_epbig} into Riesz sequences. Building upon \cite{marcus2015interlacing}, Bownik, Casazza, Marcus and Speegle also proved a quantitative version of the Feichtinger Conjecture in \cite{bownik2016improved}. 
For the specific setting of restricted Haar functions, their Corollary 6.5 in \cite{bownik2016improved} reads that if \(p>\f34\) then \eqref{eq_epbig} can be partitioned into two Riesz sequences. Theorem \ref{theo_arbeRiesz} improves on that because it already applies for \(p>\f23\) and states that \eqref{eq_epbig} is already a Riesz sequence prior to partitioning.

For more details on the relation of this paper to other work; see Section \ref{sec_remarks}.

\section{Proof of the Case \texorpdfstring{\(p>\f23\)}{p > 2/3}}


For \(n\in\mathbb{N}_0\) denote \(\D_n\seq\{I\in\D\mid |I|\geq2^{-n}\}\). The idea is to first prove a weighted inequality,
\[\Bigl\|\sum_{I\in\D_n}a_Ih_I\ind_E\Bigr\|_{L^2(w_n)}^2\geq\sum_{I\in\D_n}\|a_Ih_I\ind_E\|_2^2.\]
The weights are introduced in order to allow a proof by induction on \(n\). They will be uniformly bounded from above, so that the case \(p>\f23\) of Theorem \ref{theo_arbeRiesz} follows from the weighted inequality.

The weights \(w_n\) look as follows: Fix \(\f23<p\leq1\). Define \(g:[0,1]\rightarrow\mathbb{R}\) by
\begin{equation}\label{eq_gdef}
	g(q)\seq
	\begin{cases}
		1+\f{p(2-p)}{(3p-2)(3p-2q)}&q\geq p,\\
		g(p)\f qp&q\leq p.
	\end{cases}
\end{equation}
It is well defined because \(2q\leq 2<3p\). Now on each interval \(I\in\D_{n+1}\setminus\D_n\) abbreviate \(q_I\seq\f{|I\cap E|}{|I|}\) and assign \(w_n\) the constant value \(\f{g(q_I)}{q_I}\) if \(q_I>0\). If \(q_I=0\) then the value of \(w_n\) does not matter.
The properties of \(g\) that we need are collected in the following proposition.

\begin{pro}\label{pro_gproperties} Let \(\f23<p\leq1\). Then the function \(g\) has the following properties:
	\begin{enumerate}
		\item For all \(q_1,q_2\in[0,1]\) and \(a\in\mathbb{R}\) with \(\f{q_1+q_2}2\geq p\) or \(a=0\) we have
			\begin{equation}\label{eq_Gpos}
				\f{(1-a)^2}2g(q_1)+\f{(1+a)^2}2g(q_2)-g\Bigl(\f{q_1+q_2}2\Bigr)\geq a^2.
			\end{equation}
		\item There is a constant \(C>0\) s.t.\ for all \(q\in[0,1]\) we have
			\begin{equation}\label{eq_gcomp1}
				q\leq g(q)\leq Cq.
			\end{equation}
	\end{enumerate}
\end{pro}

Proposition \ref{pro_gproperties} is the crucial step. After that it requires not much more than bookkeeping to 
conclude the case \(p>\f23\) of Theorem \ref{theo_arbeRiesz}. In order to prove Proposition \ref{pro_gproperties}, we first show the Lemmas \ref{lem_fracconvex} and \ref{lem_g2pm1}.

Define \(\tilde g:[0,1]\rightarrow\mathbb{R}\) by
\[\tilde g(q)\seq1+\f{p(2-p)}{(3p-2)(3p-2q)},\]
so that \(g=\tilde g\) on \(q\geq p\). 

\begin{lem}\label{lem_fracconvex} 
	Let \(q_1,q_2\in[0,1],\ a\in\mathbb{R}\). Then
	\[\f{(1-a)^2}2\tilde g(q_1)+\f{(1+a)^2}2\tilde g(q_2)-\tilde g\Bigl(\f{q_1+q_2}2\Bigr)\geq a^2.\]
\end{lem}

\begin{proof} For \(i=1,2\) take 
	\(x_i\) s.t.\ 
	\begin{align*}
		\tilde g(q_i)&=1+\f1{x_i}.\\
		\intertext{Then \(x_i>0\) and}
		\tilde g\Bigl(\f{q_1+q_2}2\Bigr)&=1+\f2{x_1+x_2}.
	\end{align*}
	Hence it suffices to confirm the positivity of
	\begin{align*}
		&\f{(1-a)^2}2\Bigl(1+\f1{x_1}\Bigr)+\f{(1+a)^2}2\Bigl(1+\f1{x_2}\Bigr)-\Bigl(1+\f2{x_1+x_2}\Bigr)-a^2\\
		&=\f 12a^2\Bigl(\f1{x_1}+\f1{x_2}\Bigr)+a\Bigl(\f1{x_2}-\f1{x_1}\Bigr)+\f12\Bigl(\f1{x_1}+\f1{x_2}-4\f1{x_1+x_2}\Bigr)\\
		&=\f1{x_1x_2}\biggl[\f12a^2(x_1+x_2)+a(x_1-x_2)+\f12\Bigl(x_1+x_2-4\f{x_1x_2}{x_1+x_2}\Bigr)\biggr]
	\end{align*}
	which is a quadratic polynomial in \(a\). Since \(x_1+x_2\geq0\) and \(\f1{x_1x_2}\geq0\) it has a positive leading coefficient. The discriminant is
	\begin{align*}
		&(x_1+x_2)\Bigl(x_1+x_2-4\f{x_1x_2}{x_1+x_2}\Bigr)-(x_1-x_2)^2\\
		&=(x_1+x_2)^2-4x_1x_2-(x_1-x_2)^2=0,
	\end{align*}
	and so the minimum of the polynomial is zero.
\end{proof}

\begin{lem}\label{lem_g2pm1}
	\[g(2p-1)=\tilde g(2p-1).\]
\end{lem}

\begin{proof}
	\begin{align*}
		g(2p-1)&=g(p)\f{2p-1}p=\biggl[1+\f{p(2-p)}{p(3p-2)}\biggr]\f{2p-1}p\\
		&=\f{2p}{3p-2}\f{2p-1}p=1+\f p{3p-2}=\tilde g(2p-1).
	\end{align*}
\end{proof}

\begin{proof}[Proof of Proposition \ref{pro_gproperties}]
	Lemma \ref{lem_fracconvex} with \(a=0\) implies that \(\tilde g\) is midpoint convex and thus convex. By Lemma \ref{lem_g2pm1} and by the definition of \(g(p)\) we have that \(g(q)=\tilde g(q)\) at the two values \(q=2p-1,p\). This means that on \([0,p]\) the function \(g\) describes the line that passes through these two distinct points. On \([p,1]\) recall that \(g=\tilde g\). It follows from this that also \(g\) is convex. This means that \eqref{eq_Gpos} holds for \(a=0\) and so it suffices to consider the case 
	\(\f{q_1+q_2}2\in[p,1]\). There we have \(q_1\geq2p-q_2\geq2p-1\) and similarly \(q_2\geq2p-1\). From the considerations above we then get \(g(q_1)\geq\tilde g(q_1),\ g(q_2)\geq\tilde g(q_2),\ g\bigl(\f{q_1+q_2}2\bigr)=\tilde g\bigl(\f{q_1+q_2}2\bigr)\) which implies that for all \(a\) we have
	\begin{align*}
		&\quad\f{(1-a)^2}2g(q_1)+\f{(1+a)^2}2g(q_2)-g\Bigl(\f{q_1+q_2}2\Bigr)\\
		&\geq\f{(1-a)^2}2\tilde g(q_1)+\f{(1+a)^2}2\tilde g(q_2)-\tilde g\Bigl(\f{q_1+q_2}2\Bigr)\geq a^2,
	\end{align*}
	where the last inequality follows from Lemma \ref{lem_fracconvex}. This finishes the proof of \eqref{eq_Gpos}.

	The upper bound in \eqref{eq_gcomp1} holds for \(C=g(1)\) because \(g\) is convex and nonnegative and \(g(0)=0\). For the lower bound, recall that for \(q\in[0,p]\) we have \(g(q)=\f qpg(p)\), so that convexity implies \(g(q)\geq\f qpg(p)\) for all \(q\in[0,1]\). It follows from the definition of \(g\) that \(g(p)\geq1\) and therefore \(g(q)\geq\f qp\geq q\).
\end{proof}

The following lemma translates Proposition \ref{pro_gproperties} into our setting of Haar functions on weighted \(L^2\) spaces.
\begin{lem}\label{lem_gtoweightprop} For every \(n\in\mathbb{N}\) we have
	\begin{align}
		\label{eq_indstepmu}\Bigl\|\sum_{I\in\D_{n+1}}a_Ih_I\ind_E\Bigr\|_{L^2(w_{n+1})}^2-\Bigl\|\sum_{I\in\D_n}a_Ih_I\ind_E\Bigr\|_{L^2(w_n)}^2&\geq\sum_{I\in\D_{n+1}\setminus\D_n}\|a_Ih_I\ind_E\|_2^2\\
		\intertext{and}
		\label{eq_wsim1}1\leq w_n\leq C.
	\end{align}
	The constant \(C\) is the same as in Proposition \ref{pro_gproperties}.
\end{lem}

\begin{proof} Recall that on \(I\in\D_{n+1}\setminus\D_n\) we assign \(w_n\) the constant value \(\f{q(q_I)}{q_I}\) with \(q_I=\f{|I\cap E|}{|I|}\) if \(q_I>0\). Where \(|I\cap E|\) vanishes, the value of \(w_n\) does not matter because the integrated function in \eqref{eq_indstepmu} is zero a.e.\ on such \(I\) anyways.

	First note that \eqref{eq_wsim1} is equivalent to \eqref{eq_gcomp1}. We prove \eqref{eq_indstepmu} using mostly \eqref{eq_Gpos}. Partition the domain of integration on the left-hand-side of \eqref{eq_indstepmu} into \(\D_{n+1}\setminus\D_n\) so that the inequality becomes
	\begin{align*}
		\sum_{J\in\D_{n+1}\setminus\D_n}&\int_{E\cap J}\biggl[\Bigl(\sum_{I\in\D_{n+1}}a_Ih_I\Bigr)^2w_{n+1}-\Bigl(\sum_{I\in\D_n}a_Ih_I\Bigr)^2w_n\biggr]\\
		\geq\sum_{J\in\D_{n+1}\setminus\D_n}&\int_E(a_Jh_J)^2.
	\end{align*}
	We prove this inequality for each summand \(J\in\D_{n+1}\setminus\D_n\) individually. So fix \(J\in\D_{n+1}\setminus\D_n\). Then for each \(I\in\D_n\) the function \(h_I\) is constant on \(J\). That means we may write 
	\begin{align*}
		\ind_J\sum_{I\in\D_n}a_Ih_I&=b_J\ind_J\\
		\intertext{for some \(b_J\in\mathbb{R}\). For \(I\in\D_{n+1}\setminus\D_n\) the function \(h_I\) is nonzero if and only if \(I=J\), so that}
		\ind_J\sum_{I\in\D_{n+1}}a_Ih_I&=b_J\ind_J+a_Jh_J.
	\end{align*}
	So it suffices to show
	\begin{equation}\label{eq_indstepmuinterval} 
		\int_{E\cap J}\Bigl[(b_J\ind_J+a_Jh_J)^2w_{n+1}-(b_J\ind_J)^2w_n\Bigr]\geq\int_E(a_Jh_J)^2
	\end{equation}
	in order to prove \eqref{eq_indstepmu}.
	Write
	\begin{align*}
		q_1&\seq\f{|\lh J\cap E|}{|\lh J|},&q_2&\seq\f{|\rh J\cap E|}{|\rh J|},&\f{q_1+q_2}2&=\f{|J\cap E|}{|J|},\\
		\intertext{so that we have}
		\ind_{\lh J}w_{n+1}&=\f{g(q_1)}{q_1}\ind_{\lh J},&\ind_{\rh J}w_{n+1}&=\f{g(q_2)}{q_2}\ind_{\rh J},&\ind_Jw_n&=\f{g\bigl(\f{q_1+q_2}2\bigr)}{\f{q_1+q_2}2}\ind_J.
	\end{align*}
	if the respective denominators are positive. Evaluating the integrals, \eqref{eq_indstepmuinterval} then reads
	\[
		(b_J-a_J)^2|\lh J|g(q_1)+(b_J+a_J)^2|\rh J|g(q_2)-b_J^2|J|g\Bigl(\f{q_1+q_2}2\Bigr)\geq|J|\f{q_1+q_2}2a_J^2,
	\]
	also if \(q_1\) or \(q_2\) are zero because \(g(0)=0\). Then divide both sides by \(|J|\). For \(b_J=0\) we obtain
	\[a_J^2\f{g(q_1)+g(q_2)}2\geq\f{q_1+q_2}2a_J^2.\]
	This inequality holds due to the lower bound in \eqref{eq_gcomp1}. In case \(b_J\neq0\) we additionally divide by \(b_J^2\) and obtain
	\[\Bigl(1-\f{a_J}{b_J}\Bigr)^2\f12g(q_1)+\Bigl(1+\f{a_J}{b_J}\Bigr)^2\f12g(q_2)-g\Bigl(\f{q_1+q_2}2\Bigr)\geq\f{q_1+q_2}2\Bigl(\f{a_J}{b_J}\Bigr)^2.\]
	This inequality is a consequence of \eqref{eq_Gpos} because \(\f{q_1+q_2}2\leq1\). Note that we can also obtain the case \(b_J=0\) from \eqref{eq_Gpos} by sending \(a_J\rightarrow\infty\), instead of from \eqref{eq_gcomp1}. That way we would even get the stronger inequality without the factor \(\f{q_1+q_2}2\).
\end{proof}

\begin{proof}[Proof of Theorem \ref{theo_arbeRiesz} in case \(p>\f23\)] We use Lemma \ref{lem_gtoweightprop}. Because \(\D_0=\{[0,1)\}\) consists of only one interval we get
	\begin{equation}
		\label{eq_indstartmu}\Bigl\|\sum_{I\in\D_0}a_Ih_I\ind_E\Bigr\|_{L^2(w_0)}^2\geq\sum_{I\in\D_0}\|a_Ih_I\ind_E\|_2^2
	\end{equation}
	from the lower bound in \eqref{eq_wsim1}. 
	Summing up \eqref{eq_indstartmu} and \eqref{eq_indstepmu} for \(n=1,\ldots,k-1\) we get
	\[\Bigl\|\sum_{I\in\D_k}a_Ih_I\ind_E\Bigr\|_{L^2(w_k)}^2\geq\sum_{I\in\D_k}\|a_Ih_I\ind_E\|_2^2,\]
	and the upper bound in \eqref{eq_wsim1} allows us to get rid of the weight
	\[C\Bigl\|\sum_{I\in\D_k}a_Ih_I\ind_E\Bigr\|_2^2\geq\Bigl\|\sum_{I\in\D_k}a_Ih_I\ind_E\Bigr\|_{L^2(w_k)}^2.\]
	This proves the part \(p>\f23\) of Theorem \ref{theo_arbeRiesz} with \(c=\f1C\) if we only consider finite sums. For infinite sums where the right hand side of \eqref{eq_goal} converges, we may also pass to the limit \(n\rightarrow\infty\) with the help of \eqref{eq_calcbessel}.
\end{proof}

\section{Proof of the Case \texorpdfstring{\(p\leq\f23\)}{p <= 2/3}}\label{sec_plef23} 

Fix \(E\seq[0,\f23]\). We build the counterexample from the sequence of dyadic intervals \((I_{2n})_{n=0,1,\ldots}\), defined inductively by
\begin{align*}
	I_0&\seq[0,1],\\
	I_{2n+1}&\seq\rh{I_{2n}},\\
	I_{2n+2}&\seq\lh{I_{2n+1}}.
\end{align*}
\begin{lem}\label{lem_diswithf23}
	For all \(n=0,1,\ldots\) we have
	\begin{align*}
		|I_n|&=2^{-n},&|I_{2n}\cap E|&=\f23|I_{2n}|,&|I_{2n+1}\cap E|&=\f13|I_{2n+1}|.
	\end{align*}
\end{lem}

\begin{proof}
	It is clear that we have \(|I_n|=2^{-n}\). For the other statements we proceed by induction on \(n\). For \(n=0\) we have 
	\begin{align*}
		\biggl|[0,1]\cap\Bigl[0,\f23\Bigr]\biggr|&=\f23=\f23|[0,1]|,\\
		\biggl|\Bigl[\f12,1\Bigr]\cap\Bigl[0,\f23\Bigr]\biggr|&=\f23-\f12=\f16=\f13\biggl|\Bigl[\f12,1\Bigr]\biggr|.
	\end{align*} 
	Now assume the lemma holds for \(n\). That means that the point \(\f23\) lies \(\f13|I_{2n+1}|\) to the right from the left boundary of \(I_{2n+1}\), i.e.\ \(\f23|I_{2(n+1)}|\) to the right from the left boundary of \(I_{2(n+1)}\). Therefore we have \(|I_{2(n+1)}\cap E|=\f23|I_{2(n+1)}|\). That in turn implies that the point \(\f23\) lies \(\f16|I_{2(n+1)}|\) to the right from the midpoint of \(I_{2(n+1)}\), i.e.\(\f13|I_{2(n+1)+1}|\) to the right from the left boundary of \(I_{2(n+1)+1}\). Therefore we have \(|I_{2(n+1)+1}\cap E|=\f13|I_{2(n+1)+1}|\), finishing the proof of the lemma for \(n+1\).
\end{proof}

Further set
\begin{align*}
	a_0&\seq1,\\
	a_{2n}&\seq2^{n-1},\qquad n\geq1.
\end{align*}
The following proposition proves the case \(p\leq\f23\) of Theorem \ref{theo_arbeRiesz}.

\begin{pro}\label{pro_counter}
	For each \(n\) we have
	\begin{align}
		\label{eq_Asimn}\sum_{k=0}^n\|a_{2k}h_{I_{2k}}\ind_E\|_2^2&=\f23+\f n6,\\
		\label{eq_Bconst}\Bigl\|\sum_{k=0}^na_{2k}h_{I_{2k}}\ind_E\Bigr\|_2^2&=\f23.
	\end{align}
\end{pro}

\begin{proof} 
	By Lemma \ref{lem_diswithf23} we have \(|I_{2n}\cap E|=\f232^{-2n}\). Thus
	\begin{align*}
		\|a_0h_{I_0}\ind_E\|_2^2&=\f23,\\
		\|a_{2n}h_{I_{2n}}\ind_E\|_2^2&=\f232^{-2n}2^{2(n-1)}=\f16,\qquad n\geq1,
	\end{align*}
	which implies \eqref{eq_Asimn}.
	
	In order to prove \eqref{eq_Bconst}, first note that the support of \(h_{I_{2(n+1)}}\ind_E\) is \(\rh{I_{2n}}\cap E\). Therefore it follows by induction on \(n\) that
	\[\sum_{k=0}^na_{2k}h_{I_{2k}}\ind_E=-\ind_{[0,\f12)}+2^n\ind_{\rh{I_{2n}}\cap E},\]
	since
	\[2^n\ind_{\rh{I_{2n}}\cap E}+2^nh_{I_{2(n+1)}}\ind_E=2^{n+1}\ind_{\rh{I_{2(n+1)}}\cap E}.\] 
	By Lemma \ref{lem_diswithf23} we have \(|\rh{I_{2n}}\cap E|=\f13|\rh{I_{2n}}|=\f132^{-2n-1}\) so that we obtain
	\[\Bigl\|\sum_{k=0}^na_{2k}h_{I_{2k}}\ind_E\Bigr\|_2^2=\f12\cdot(-1)^2+\f132^{-2n-1}2^{2n}=\f23.\]
\end{proof}

\section{Remarks}\label{sec_remarks}

\subsection{Optimality of the Constant \(c\)}

From the proof we get an explicit expression for the constant in \eqref{eq_goal}
\[c=\f1 C=\f1{g(1)}=\f{(3p-2)^2}{(3p-2)^2+p(2-p)}=\f{\bigl(p-\f23\bigr)^2}{\bigl(p-\f23\bigr)^2+\f19p(2-p)}\]
and since
\[p(2-p)=(p-\f23)\biggl[\f43-(p-\f23)\biggr]+\f23\biggl[\f43-(p-\f23)\biggr]=\f89+\ord\Bigl(p-\f23\Bigr)\]
we have
\[c=\f{\bigl(p-\f23\bigr)^2}{\f8{81}+\ord(p-\f23)}=\f{81}8\Bigl(p-\f23\Bigr)^2\f1{1+\ord(p-\f23)}=\f{81}8\Bigl(p-\f23\Bigr)^2+\ord\Bigl(p-\f23\Bigr)^3.\] 
However this constant \(c\) is likely not maximal because \(g\) satisfies a stronger bound than the required \(g(q)\geq q\), and because we dropped a factor \(\f{q_1+q_2}2\) in the deduction of \eqref{eq_Gpos}. We only did this because sending \(q_1\rightarrow q_2\) in \eqref{eq_Gpos} leads to an ODE with solution \(\tilde g\), while with the factor \(\f{q_1+q_2}2\) in place we could not solve the ODE. We did however minimize \(C\) in some respect: There are multiple solutions to the ODE from \eqref{eq_Gpos} such that the corresponding \(g\) satisfies \eqref{eq_gcomp1} and \eqref{eq_Gpos} with some \(C\). Among all those, \(\tilde g\) has the smallest \(C\) for \(p\rightarrow\f23\). For a proof of this and for more details; see Section 2.1 in \cite{thesis}. 
In Section 3 in \cite{thesis} we also provide a set \(E\) for which we prove an explicit sharp constant \(c_{\tx{s}}\) which satisfies \(c_{\tx{s}}=27\bigl(p-\f23)^2+\ord(p-\f23)^3\). 
We conjecture that \eqref{eq_goal} holds already with this particular constant \(c_{\tx{s}}\). In Section 4 of \cite{thesis} we prove that this is indeed the case at least for certain \(E\).

\subsection{Further remarks on Theorem \ref{theo_arbeRiesz}}

For \(p>\f23\) inequality \eqref{eq_goal} still holds if we add the constant function to the sums, i.e.\ allow \(a_{[0,2)}\neq0\), even though usually \(|[0,2)\cap E|<p|[0,2)|\). 

Furthermore, Theorem \ref{theo_arbeRiesz} is not a consequence of the fact that \(\bigl\{h_I\ind_E\bigm||I\cap E|\geq p|I|\bigr\}\) is only a small perturbation of the orthogonal set \(\bigl\{h_I\bigm||I\cap E|\geq p|I|\bigr\}\), in the sense that 
\[\|h_I-h_I\ind_E\|_2^2\leq(1-p)\|h_I\|_2^2<\f13\|h_I\|_2^2.\]
In order to see this, consider the following example. Assume that \(u_1,\ldots,u_n\) are orthonormal. Abbreviate \(u\seq u_1+\ldots+u_n\) and for \(i=1,\ldots,n\) set
\[u_i'\seq u_i-\f1n u.\]
Then
\[\|u_i-u_i'\|^2=\f1{n^2}\|u\|^2=\f1{n^2}\sum_{i=1}^n\|u_i\|^2=\f1n.\]
but
\[\|u_1'+\ldots+u_n'\|^2=\|u-u\|^2=0\not\gtrsim\|u_1'\|^2+\ldots+\|u_n'\|^2.\]

\subsection{Related Topics}

The starting point of this work was the following question, because its answer could provide ideas on how to to extend the result in \cite{kovac_thiele_zorin-kranich_2015} to three general functions.
\begin{que}\label{que_orig}
	Let \(\D\) be the set of dyadic intervals of \([0,1)\). Let \([0,1)=E_0\cup E_1\) be a partition. Is there a partition \(\D=D_0\cup D_1\) such that for \(i=0,1\) the equivalence
	\begin{equation}\label{eq_sequenceorig}
		\Bigl\|\sum_{I\in D_i}a_Ih_I\ind_{E_i}\Bigr\|_2^2\sim\sum_{I\in D_i}\|a_Ih_I\ind_{E_i}\|_2^2
	\end{equation}
	is true?
\end{que}

We started investigating this question in Section 5 in \cite{thesis}. An initial approach to Question \ref{que_orig} could be to construct a partition by a majority decision: For \(i=0,1\) take \(D_i\) s.t.\ for all \(I\in D_i\) we have 
\begin{equation}\label{eq_majoritydecision}
	|I\cap E_i|\geq\f12|I|.
\end{equation}
However by the counterexample of Theorem \ref{theo_arbeRiesz} for \(p=\f23\geq\f12\), this strategy does not result in the lower bound in \eqref{eq_sequenceorig}. 
 Although by \eqref{eq_calcbessel} the majority decision \eqref{eq_majoritydecision} at least leads to the upper bound in \eqref{eq_sequenceorig}. Another idea was to use the Feichtinger-Conjecture which was recently resolved by Marcus, Spielman and Srivastava in \cite{marcus2015interlacing}. Based on \cite{marcus2015interlacing}, Bownik, Casazza, Marcus and Speegle proved a quantified version of the Feichtinger Conjecture in \cite{bownik2016improved}. The following theorem is a reformulation of Corollary 6.5 in \cite{bownik2016improved}.

\begin{theo}\label{theo_bcms}
	Let \(C<\f43\) and \(c\seq\f C2-\sqrt{2(C-1)(2-C)}\). Let \(V\) be a sequence in a Hilbert space such that for all \((a_v)_{v\in V}\subset\mathbb{R}\)
	\[\Bigl\|\sum_{v\in V}a_vv\Bigr\|^2\leq C\sum_{v\in V}\|a_vv\|^2.\]
	Then there is a partition \(V=V_0\cup V_1\) such that for \(i=0,1\) we have
	\[\Bigl\|\sum_{v\in V_i}a_vv\Bigr\|^2\geq c\sum_{v\in V_i}\|a_vv\|^2.\]
\end{theo}

Unfortunately it is not clear if Theorem \ref{theo_bcms} can be used to answer Question \ref{que_orig}. The closest consequence of Theorem \ref{theo_bcms} in that direction that we found is the following corollary.

\begin{cor}\label{cor_bcmshere}
	Let \(p>\f34\) and \(c\seq\f1{2p}-\sqrt{2(\f1p-1)(2-\f1p)}\). Let \(E\subset[0,1)\) and \(E=E_0\cup E_1\) be a partition and for \(i=0,1\) set
	\[
		H_i\seq\bigl\{h_I\ind_{E_i}\bigm| I\in\D,\ |I\cap E_i|\geq p|I|\bigr\}.
	\]
	Then \(H_0\cup H_1\) can be partitioned into \(G_0\cup G_1\) where for \(i=0,1\) we have
	\begin{equation}\label{eq_Griesz}
		\Bigl\|\sum_{v\in G_i}a_vv\Bigr\|^2\geq c\sum_{v\in G_i}\|a_vv\|^2.
	\end{equation}
\end{cor}

Theorem \ref{theo_arbeRiesz} can be seen as an improvement of Corollary \ref{cor_bcmshere}. That is because by Theorem \ref{theo_arbeRiesz} the two sequences \(H_0\) and \(H_1\) already satisfy \eqref{eq_Griesz} with some other constant \(c>0\), and since \(H_0\) and \(H_1\) are orthogonal to one another, also their union satisfies \eqref{eq_Griesz}, even without partitioning. And this already holds for \(p>\f23\).


\nocite{*}
\bibliographystyle{plain}
\bibliography{bib}

\end{document}